\documentclass[12pt]{amsart}

\usepackage{amsthm, amsmath}
\usepackage{amssymb}
\usepackage[all]{xy}
\usepackage{mathrsfs}
\usepackage[latin1]{inputenc}
\usepackage{graphicx}
\numberwithin{equation}{subsection}

\DeclareMathOperator{\Flag}{Flag}
\DeclareMathOperator{\Grass}{Grass}
\DeclareMathOperator{\defect}{def}
\DeclareMathOperator{\supp}{supp}

\theoremstyle{plain}

\newtheorem{theorem}{Theorem}[subsection]

\newtheorem{lemma}[theorem]{Lemma}

\newtheorem{cor}[theorem]{Corollary}

\newtheorem{defn}[theorem]{Definition}
\newtheorem{prop}[theorem]{Proposition}

\newtheorem{conj}[theorem]{Conjecture}

\theoremstyle{definition}

\newtheorem{example}[theorem]{Example}

\def\ge{\geqslant}
\def\le{\leqslant}
\def\a{\alpha}

\def\g{\gamma}

\def\d{\delta}

\def\e{\epsilon}

\def\s{\sigma}
\def\t{\tau}
\def\th{\theta}
\def\k{\kappa}
\def\l{\lambda}

\def\i{^{-1}}

\def\NN{\mathbb N}

\def\co{\mathcal O}

\def\tW{\tilde W}
\def\tw{\tilde w}
\def\tS{\tilde S}

\usepackage{amssymb} 
\usepackage{latexsym} 
\usepackage{amsfonts} 
\usepackage{amsmath} 
\usepackage{eucal} 
\usepackage{bm} 
\usepackage{bbm} 
\usepackage{graphicx} 
\usepackage[english]{varioref} 
\usepackage[nice]{nicefrac} 
\usepackage[all]{xy}

\newcommand{\kk}{\Bbbk}

\def\<{\langle}
\def\>{\rangle}

\begin{document}

\title[Dimensions of affine Deligne-Lusztig varieties]
{Dimensions of affine Deligne-Lusztig\\varieties in affine flag varieties}

\author[U. G\"{o}rtz]{Ulrich G\"{o}rtz}
\address{Ulrich G\"{o}rtz\\Institut f\"ur Experimentelle Mathematik\\Universit\"at Duisburg-Essen\\Ellernstr.~29\\45326 Essen\\Germany}
\email{ulrich.goertz@uni-due.de}
\thanks{G\"{o}rtz was partially supported by a Heisenberg grant and by the Sonderforschungsbereich TR 45 ``Periods, Moduli spaces and Arithmetic of Algebraic Varieties'' of the Deutsche Forschungsgemeinschaft.}
 
\author[X. He]{Xuhua He}
\address{Xuhua He, Department of Mathematics, The Hong Kong University of Science and Technology, Clear Water Bay, Kowloon, Hong Kong}
\thanks{Xuhua He was partially supported by HKRGC grants 601409.}
\email{maxhhe@ust.hk}

\begin{abstract} 
Affine Deligne-Lusztig varieties are analogs of Deligne-Lusztig varieties in the context of an affine root system. We prove a conjecture stated in the paper \cite{GHKR2} by Haines, Kottwitz, Reuman, and the first named author, about the question which affine Deligne-Lusztig varieties (for a split group and a basic $\sigma$-conjugacy class) in the Iwahori case are non-empty. If the underlying algebraic group is a classical group and the chosen basic $\sigma$-conjugacy class is the class of $b=1$, we also prove the dimension formula predicted in op.~cit.~in almost all cases.
\end{abstract}

\maketitle

\section{Introduction}

\subsection{}
Affine Deligne-Lusztig varieties, which are analogs of usual De\-li\-gne-Lusztig varieties~\cite{DL} in the context of an affine root system, have been studied by several people, mainly because they encode interesting information about the reduction of Shimura varieties and specifically about the relation between the ``Newton stratification'' and the ``Kottwitz-Rapoport stratification''. Their definition is purely group-theoretical. To recall it, we fix a split connected reductive group over the finite field $\mathbb F_q$ with $q$ elements.  Let $\kk$ be an algebraic closure of $\mathbb F_q$, let $L=\kk((\e))$ be the field of formal Laurent series over $\kk$, and let $\s$ be the automorphism of $L$ defined by $\s(\sum a_n \e^n)=\sum a_n^q \e^n$. We also denote the induced automorphism on the loop group $G(L)$ by $\s$. Let $T\subset G$ be a split maximal torus, and denote by $W$ the corresponding Weyl group. Furthermore, let $I\subset G(\kk[ [\e]])$ be an Iwahori subgroup containing $T(\kk[ [\e ]])$, and let  $\tW$ be the extended affine Weyl group attached to these data. See Section~\ref{sec:notation} for details.

For $x\in\tW$ and $b\in G(L)$ the locally closed subscheme
\[
X_x(b) = \{ g\in G(L)/I;\ g^{-1}b\s(g) \in IxI \}
\]
of the affine flag variety $G(L)/I$ is called the \emph{affine Deligne-Lusztig variety} attached to $b$ and $x$. The $X_x(b)$ are known to be finite-dimensional varieties (locally of finite type over $\kk$), but are possibly empty, and it is not in general easy to check whether $X_x(b)=\emptyset$ for a given pair $x$, $b$.

In~\cite[Conjecture~9.5.1~(a)]{GHKR2} , Haines, Kottwitz, Reuman and the first named author have stated the following conjecture, which extends a conjecture formulated earlier by Reuman. For simplicity, let us assume that $G$ is quasi-simple of adjoint type.

We denote by $\tW'$ the lowest two-sided cell in the sense of Kazhdan and Lusztig. In the terminology of \cite{GHKR2}, this is the union of the shrunken Weyl chambers. See Section~\ref{sec:defW'}. The notion of \emph{basic} $\sigma$-conjugacy class can be characterized by saying that it contains an element of $N(T)(L)$ which gives rise to a length $0$ element of $\tW = N(T)(L)/T(\kk[ [\e]])$. Equivalently, a $\sigma$-conjugacy class is basic if and only if its Newton vector is central. See~\cite[Lemma 7.2.1]{GHKR2}. The $\sigma$-conjugacy class of $b=1$ is always basic.

\begin{conj}
Suppose that the $\sigma$-conjugacy class of $b$ is basic, and that $x\in \tW'$. If $b$ and $x$ are in the same connected component of $G(L)$ and
\[
\eta(x) \in W\setminus \bigcup_{T\subsetneq S} W_T,
\]
then $X_x(b)\ne \emptyset$ and
\[
\dim X_x(b) = \frac12 \left( \ell(x) + \ell(\eta(x)) - {\rm def}_G(b)  \right).
\]
\end{conj}

Here $S$ is the set of simple reflections, and for any $T\subset S$, $W_T$ denotes the subgroup of $W$ generated by $T$. Furthermore, $\eta$ is the map $\tW\to W$ given as follows: If $x = vt^\mu w$ with $v,w\in W$ and such that the alcove $t^\mu w$ lies in the dominant chamber, then $\eta(x) = wv$. See Section~\ref{sec:outline}. Finally, ${\rm def}_G(b)$ is the defect of $b$, see \cite{Kottwitz-defect}.

A strengthened version of the converse of the non-emptiness statement was proved in \cite[Proposition 9.5.4]{GHKR2}. Here we prove
\begin{theorem}
Suppose that the $\sigma$-conjugacy class of $b$ is basic, and that $x\in \tW'$. Write $x=vt^\mu w$ as above. Assume that $b$ and $x$ are in the same connected component of $G(L)$ and that 
\[
\eta(x) \in W\setminus \bigcup_{T\subsetneq S} W_T.
\]
\begin{enumerate}
\item
Then $X_x(b)\ne \emptyset$.
\item
If $v=w_0$ or $\mu$ is regular, then
\[
\dim X_x(b) \le \frac12 \left( \ell(x) + \ell(\eta(x)) - {\rm def}_G(b)  \right).
\]
\item
If $G$ is a classical group and $b=1$, or of type $A_n$ and $b$ arbitrary basic, then
\[
\dim X_x(b) \ge \frac12 \left( \ell(x) + \ell(\eta(x)) - {\rm def}_G(b)  \right).
\]
\end{enumerate}
\end{theorem}

See Section~\ref{sec:outline} for more detailed statements and an outline of the proof. Roughly speaking, our methods are combinatorial (whereas in \cite{GHKR2} the theory of $\epsilon$-adic groups was used). Important ingredients are a refinement of the reduction method of Deligne and Lusztig (Section~2), and the results of the second named author about conjugacy classes in affine Weyl groups, see~\cite{He}.

In~\cite{Beazley}, E.~Beazley obtained similar results for groups of type $A_n$, $C_2$ or $G_2$ using a similar method, but using only results on conjugacy classes in finite Weyl groups.

In Section~4, we briefly consider the case that $x\in\tW\setminus \tW'$, but all in all this case remains unclear. Note however that the relation to stratifications of the wonderful compactification of $G$ might provide further insight in this case, see~\cite{He-WonderfulComp}. Finally, a careful study of the reduction method also shows that affine Deligne-Lusztig varieties in the affine flag varieties are not equidimensional in general; we give a specific example in Section~5. Note that for affine Deligne-Lusztig varieties in the affine Grassmannian, equidimensionality if known if $b$ is in the torus $T(L)$ (\cite[Proposition 2.17.1]{GHKR1}) and if $b$ is basic (\cite[Theorem 1.2]{HartlViehmann}); in the intermediate cases, it is still an open question.


\section{Preliminaries}

\subsection{Notation}\label{sec:notation}

Let $G$ be a split connected reductive group over $\mathbb F_q$. We assume that $G$ is quasi-simple of adjoint type. As explained in \cite[5.9]{GHKR1}, all problems about the dimension of affine Deligne-Lusztig varieties easily reduce to this case.

Let $L=\kk((\e))$ be the field of formal Laurent series over $\kk$, and let $\s$ be the automorphism on $L$ defined by $\s(\sum a_n \e^n)=\sum a_n^q \e^n$. We also denote the induced automorphism on the loop group $G(L)$ by $\s$. 

Let $T$ be a maximal torus of $G$, let $B\supset T$ be a Borel subgroup of $G$, and let $B^-$ be the opposite Borel subgroup so that $T=B \cap B^-$. Let $\Phi$ be the set of roots and $Y$ be the coweight lattice. We denote by $Y_+$ the set of dominant coweights. Let $(\a_i)_{i \in S}$ be the set of simple roots determined by $(B, T)$. We denote by $W$ the Weyl group $N(T)/T$. For $i \in S$, we denote by $s_i$ the simple reflection corresponding to $i$.

For $w \in W$, we denote by $\supp(w)$ the set of simple reflections occurring in a reduced expression of $w$. So the condition $w\in W\setminus \bigcup_{T\subsetneq S}W_T$ is equivalent to $\supp(w)=S$.

For $w \in W$, we choose a representative in $N(T)$ and also write it as $w$. For any $J \subset S$, let $\Phi^+_J$ (resp. $\Phi^-_J$) be the positive (resp.~negative) roots spanned by $(\a_j)_{j \in J}$. 

Let $I$ be the inverse image of $B^-$ under the projection map $G(\kk[[\e]]) \mapsto G$ sending $\e$ to $0$. Let $\tW=N(T(L))/(T(L) \cap I)$ be the extended affine Weyl group of $G(L)$. Then it is known that $\tW=W \ltimes Y=\{w \e^\chi; w \in W, \chi \in Y\}$. Let $\ell: \tW \to \NN \cup \{0\}$ be the length function. For $x=w \e^\chi \in \tW$, we also write $x$ for the representative $w \e^\chi$ in $N(T(L))$. 

Let $X$ be the coroot lattice, and let $W_a=W \ltimes X \subset \tW$ be the affine Weyl group. Set $\tS=S \cup \{0\}$ and $s_0=\e^{\th^\vee} s_\th$, where $\th$ is the largest positive root of $G$. Then $(W_a, \tS)$ is a Coxeter system. Let $\k: \tW \to \tW/W_a$ be the natural projection. 

For any $J \subset \tS$, let $W_J$ be the subgroup of $W_a$ generated by $J$ and $\tW^J$ (resp. ${}^J \tW$) be the set of minimal length coset representative of $\tW/W_J$ (resp. $W_J \backslash \tW$). For example, ${ }^S \tW$ is the set of all elements for which the corresponding alcove is contained in the dominant chamber. In the case where $J \subset S$, we write $W^J$ for $\tW^J \cap W$ and ${}^J W$ for ${}^J \tW \cap W$.

Let $\l$ be a dominant coweight. Set $I(\l)=\{i \in S; \<\l, \a_i\>=0\}$, the ``set of walls'' that $\l$ lies on. For $J \subset S$, let $\rho^\vee_J \in Y_+$ with
\[
\<\rho^\vee_J, \a_i\>=\begin{cases} 1, & \text{ if } j \in J \\ 0, & \text{ if } j \notin J\end{cases}.
\]
We simply write $\rho^\vee$ for $\rho^\vee_S$. 

For any root $\a \in \Phi$, set $\d_\a=\begin{cases} 1, & \text{ if } \a \in \Phi^- \\ 0, & \text{ if } \a \in \Phi^+ \end{cases}$.

\subsection{}

Following \cite[1.4]{He}, we use the following notation:
For $x, x' \in \tW$ and $i \in \tS$, we write $x \xrightarrow{s_i} x'$ if $x'=s_i x s_i$ and $\ell(x') \le \ell(x)$. We write $x \tilde \to x'$ if there is a sequence $x=x_0, x_1, \cdots, x_n=x'$ of elements in $\tW$ such that for all $k$, $x_k=\t x_{k-1} \t \i$ for some $\t \in \tW$ with $\ell(\t)=0$ or $\tw_{k-1} \xrightarrow{s_i}_{\d} \tw_k$ for some $i \in \tS$. 
We write $x \tilde \approx x'$ if $x \tilde \to x'$ and $x' \tilde \to x$.


\subsection{}\label{sec:defW'}
Any element in $\tW$ can be written in a unique way as $v t^\mu w$ for $\mu \in Y_+$, $v \in W$ and $w \in {}^{I(\mu)} W$.
Note that in this case $t^\mu w \in {^S}\tW$, and $\ell(v \t^\mu w) = \ell(v) + \ell(t^\mu) - \ell(w)$.
Set
\[
\tW'=\{v t^\mu w;\ \mu\in Y_+,\ w\in{}^{I(\mu)} W,\  \<\mu, \a_i\>+\d_{v \a_i}-\d_{w \i \a_i} \neq 0 \quad \forall i \in S\}.
\]

It is proved by Lusztig \cite{L}, Shi \cite{S} and B\'edard \cite{B} that $\tW' \cap C$ is a two-sided cell for each $W_a$-coset $C$ in $\tW$. It is called the {\it lowest two-sided cell}. It is also called the union of the {\it shrunken Weyl chambers} in \cite{GHKR1} and \cite{GHKR2}.


\subsection{}
\label{subsec2.4}

We introduce a convenient notation for varieties of tuples of elements in $\Flag = G(L)/I$ (i.e., the affine flag variety of $G$ over $\kk$). Instead of giving a rigorous definition, it is more useful to explain the notation by examples. We denote by $\co_w \subset \Flag\times\Flag$ the locally closed subvariety of pairs $(g,g')$ such that the relative position of $g$ and $g'$ is $w$. Then we set
\begin{align*}
& \{ \xymatrix{ g \ar[r]^-w & g'' \ar[r]^-{w'} & g' } \} := \\
& \{  (g, g', g'') \in (\Flag)^3;\ (g,g'')\in \co_{w},\ (g'',g')\in \co_{w'} \}.
\end{align*}
Similarly,
\begin{align*}
& \{ \xymatrix{ g \ar[r]^-w \ar@/_2pc/[rr]_-{w''} & g'' \ar[r]^-{w'} & g' } \} := \\
& \{ ( g, g', g'') \in (\Flag)^3;\ (g,g'')\in \co_{w},\ (g'',g')\in \co_{w'},\ (g,g')\in\co_{w''} \}.
\end{align*}
Finally, we need conditions on relative positions where elements $g$ and $b\sigma(g)$ occur both---the simplest case being the affine Deligne-Lusztig varieties themselves:
\[
X_x(b) = \{ \xymatrix{ g \ar[r]^-x & b\s(g) } \}.
\]
In all these cases, we do not distinguish between the sets given by the conditions on the relative position, and the corresponding locally closed sub-ind-schemes of the product of affine flag varieties. 

The following properties are easy to prove. 

(1) Let $x, y \in \tW$. If $l(x y)=l(x)+l(y)$, then the map $(g, g', g'') \mapsto (g, g')$ gives an isomorphism \[\{ \xymatrix{ g \ar[r]^-x \ar@/_2pc/[rr]_-{x y} & g'' \ar[r]^-{y} & g' } \} \to \{ \xymatrix{ g \ar[r]^-{x y} & g' } \}.\]

(2) Let $w \in \tW$ and $s \in \tS$. If $w s<w$, then 
\[
\{\xymatrix{g \ar[r]^-w & g'' \ar[r]^-s & g'}\}=\{ \xymatrix{ g \ar[r]^-w \ar@/_2pc/[rr]_-{w s} & g'' \ar[r]^-s & g' } \} \sqcup \{ \xymatrix{ g \ar[r]^-w \ar@/_2pc/[rr]_-w & g'' \ar[r]^-s & g' } \},
\] 
where the first set on the right hand side of the equation is open, and the second one is closed. The projections $(g, g', g'') \mapsto (g, g')$ give rise to Zariski-locally trivial fiber bundles
\begin{align*} 
& \{ \xymatrix{ g \ar[r]^-w \ar@/_2pc/[rr]_-{w s} & g'' \ar[r]^-s & g' } \} \to \{ \xymatrix{ g \ar[r]^-{w s} & g' } \}; \\ 
&\{ \xymatrix{ g \ar[r]^-w \ar@/_2pc/[rr]_-w & g'' \ar[r]^-s & g' } \} \to \{ \xymatrix{ g \ar[r]^-w & g' } \}. 
\end{align*}
with fibers isomorphic to $\mathbb A^1$ in the first case, and isomorphic to $\mathbb A^1\setminus \{0\}$ in the second case.

\subsection{The reduction method of Deligne and Lusztig}

\begin{lemma}[He \cite{H1}, Lemma 1]
Let $w, w'\in \tW$. Then the set $\{ uu';\ u\le w,\ u'\le w' \}$ has a unique maximal element, which we denote by $w*w'$. We have $\ell(w*w') = \ell(w) + \ell(w\i (w*w')) = \ell( (w*w')(w')\i) + \ell(w')$, and $\supp(w*w') = \supp(w) \cup \supp(w')$.
\end{lemma}

Note that the operation $*$ is associative. 

\begin{prop}\label{generalized-DL-reduction}
Let $w, w'\in \tW$, and let $w'' \in \{ ww', w*w' \}$. All fibers of the projection
\[
\pi\colon\{ \xymatrix{ g \ar[r]^-w \ar@/_2pc/[rr]_-{w''} & g'' \ar[r]^-{w'} & g' } \}
\longrightarrow
\{ \xymatrix{ g \ar[r]^-{w''} & g'} \}
\]
which maps $(g, g', g'')$ to $(g, g')$ have dimension
\[
\dim\pi^{-1}( ( g, g') ) \ge \left\{
\begin{array}{ll}
\ell(w) + \ell(w') - \ell(w*w') & \text{if } w'' = w*w', \\
\frac 12  ( \ell(w) + \ell(w') - \ell(ww') ) & \text{if } w'' = ww'.
\end{array}
\right.
\]
\end{prop}


\begin{proof}
We proceed by induction on $\ell(w')$. If $w'=1$, then the statement is obvious. Now assume that $l(w')>0$. Then $w'=s w'_1$ for some $s \in \tS$ and $w'_1 \in \tW$ with $\ell(w')=\ell(w'_1)+1$. Then 
\[
\{ \xymatrix{ g \ar[r]^-w \ar@/_2pc/[rrr]_-{w''} & g'' \ar[r]^-s & g''' \ar[r]^-{w'_1} & g' } \} \cong 
\{ \xymatrix{ g \ar[r]^-w \ar@/_2pc/[rr]_-{w''} & g'' \ar[r]^-{w'} & g' } \}
\]

If $w s>w$, then 
\[
\{ \xymatrix{ g \ar[r]^-w \ar@/_2pc/[rrr]_-{w''} & g'' \ar[r]^-s & g''' \ar[r]^-{w'_1} & g' } \} \cong 
\{ \xymatrix{ g \ar[r]^-{w s} \ar@/_2pc/[rr]_-{w''} & g''' \ar[r]^-{w'_1} & g' }. \}
\]
We also have that $\ell(w s)+\ell(w'_1)=\ell(w)+1+\ell(w')-1=\ell(w)+\ell(w')$, $w s w'_1=w w'$ and $(w s)*w'_1=w*s*w'_1=w*w'$. Now the statement follows from inductive hypothesis on $w'_1$. 

If $w s<w$, then 
\[
\{ \xymatrix{ g \ar[r]^-w \ar@/_2pc/[rrr]_-{w''} & g'' \ar[r]^-s & g''' \ar[r]^-{w'_1} & g' } \}=X_1 \sqcup X_2,
\] 
where 
\begin{align*}
X_1 &=\{\xymatrix{ g \ar[r]^-w \ar@/_2pc/[rrr]_-{w''} \ar@/^2pc/[rr]^-{ws} & g'' \ar[r]^-s & g''' \ar[r]^-{w'_1} & g' } \} \\ 
X_2 &=\{\xymatrix{ g \ar[r]^-w \ar@/_2pc/[rrr]_-{w''} \ar@/^2pc/[rr]^-w & g'' \ar[r]^-s & g''' \ar[r]^-{w'_1} & g' } \}.
\end{align*}

The projection from $X_1$ to $\{ \xymatrix{ g \ar[r]^-{w''} & g'} \}$ factors through
\[ X_1 \to \{\xymatrix{ g \ar@/_2pc/[rr]_-{w''} \ar[r]^-{ws} & g''' \ar[r]^-{w'_1} & g' } \} \to \{ \xymatrix{ g \ar[r]^-{w''} & g'} \},\] where the first map is a bundle map whose fibers are all of dimension $1$. If $w''=w w'$, then by induction hypothesis on $w'_1$, the fibers of the second map all have dimension $\ge \frac{1}{2}(\ell(w s)+\ell(w'_1)-l(w w'))$. Notice that $\ell(w s)+\ell(w'_1)=\ell(w)-1+\ell(w')-1=\ell(w)+\ell(w')-2$. So the fibers of the map $X_1 \to \{ \xymatrix{ g \ar[r]^-{w''} & g'} \}$ all have dimension $\ge \frac{1}{2}(\ell(w)+\ell(w')-l(w w'))$, which gives us the desired lower bound on the fiber dimension.

The projection from $X_2$ to $\{ \xymatrix{ g \ar[r]^-{w''} & g'} \}$ factors through
\[ X_2 \to \{\xymatrix{ g \ar@/_2pc/[rr]_-{w''} \ar[r]^-w & g''' \ar[r]^-{w'_1} & g' } \} \to \{ \xymatrix{ g \ar[r]^-{w''} & g'} \},\] where the first map is a bundle map whose fibers are all of dimension $1$. If $w''=w*w'=w*s*w'_1=w*w'_1$, then by induction hypothesis on $w'_1$, the fibers of the second map all have dimension $\ge (\ell(w)+\ell(w'_1)-\ell(w*w'))$. Notice that $\ell(w)+\ell(w'_1)=\ell(w)+\ell(w')-1$. So the fibers of the map $X_2 \to \{ \xymatrix{ g \ar[r]^-{w''} & g'} \}$ all have dimension $\ge (\ell(w)+\ell(w')-\ell(w*(w'))$.
\end{proof}

As a corollary, we can prove the analog, in the affine context, of the ``reduction method'' of Deligne and Lusztig (see \cite[proof of Theorem 1.6]{DL}). This result can of course be proved directly, along the lines of the proof of the proposition above, and was also worked out before by Haines at the suggestion of Lusztig.

\begin{cor}\label{Lemma-DLReduction}
Let $x\in \widetilde{W}$, and let $s\in\tS$ be a simple affine reflection.
\begin{enumerate}
\item 
If $\ell(sxs) = \ell(x)$, then there exists a universal homeomorphism $X_x(b) \rightarrow X_{sxs}(b)$.
\item
If $\ell(sxs) = \ell(x)-2$, then $X_x(b)$ can be written as a disjoint union $X_x(b) = X_1 \sqcup X_2$ where $X_1$ is closed and $X_2$ is open, and such that there exist morphisms $X_1\to X_{sxs}(b)$ and $X_2\to X_{sx}(b)$ which are compositions of a Zariski-locally trivial fiber bundle with one-dimensional fibers and a universal homeomorphism.
\end{enumerate}
\end{cor}

\begin{proof}
By possibly exchanging, in case~(1), $x$ and $sxs$, we may assume that $sx < x$. By the proposition, the projection
\[
X' :=
\left\{
\raisebox{.5\depth}{
\xymatrix{ g \ar[d]_-s \ar[r]^-x & b\sigma(g) \ar[d]^-s \\
g_1 \ar[ur]^-{sx} & b\sigma(g_1)
}
}
\right\}
\longrightarrow
\{ \xymatrix{ g \ar[r]^-x & b\sigma(g) } \} = X_x(b).
\]
is an isomorphism, so we may replace $X_x(b)$ by $X'$. We write $X'$ as the disjoint union
\[
X' = X_1 \sqcup X_2 :=
\left\{
\raisebox{.5\depth}{
\xymatrix{ g \ar[d]_-s \ar[r]^-x & b\sigma(g) \ar[d]^-s \\
g_1 \ar[ur]^-{sx} \ar[r]_-{sxs} & b\sigma(g_1)
}
}
\right\}
\sqcup
\left\{
\raisebox{.5\depth}{
\xymatrix{ g \ar[d]_-s \ar[r]^-x & b\sigma(g) \ar[d]^-s \\
g_1 \ar[ur]^-{sx} \ar[r]_-{sx} & b\sigma(g_1)
}
}
\right\}
\]
Since we have $\ell(sx)<\ell(x)$, the natural morphism
\[
X_1 \rightarrow X_1' = \{ \xymatrix{ g_1 \ar[r]^-{sx} \ar@/_2pc/[rr]_-{sxs} & g_2 \ar[r]^-{s} & b\sigma(g_1) } \}
\]
is a universal homeomorphism (note that the composition of $X_1\rightarrow X_1'$ with the projection to $g_2$ is the map $g\mapsto b\sigma(g)$).

Now we distinguish between the two cases. In case~(1), $X_2 = \emptyset$, and applying the proposition once more, we find that in this case the projection
\[
X_1' = \{ \xymatrix{ g_1 \ar[r]^-{sx} \ar@/_2pc/[rr]_-{sxs} & g_2 \ar[r]^-{s} & b\sigma(g_1) } \}
\longrightarrow 
\{ \xymatrix{ g_1 \ar[r]^-{sxs} & b\sigma(g_1)  }\} = X_{sxs}(b)
\]
is an isomorphism.

Next we come to case~(2). The projection
\[
X_1' = \{ \xymatrix{ g_1 \ar[r]^-{sx} \ar@/_2pc/[rr]_-{sxs} & g_2 \ar[r]^-{s} & b\sigma(g_1) } \}
\longrightarrow 
\{ \xymatrix{ g_1 \ar[r]^-{sxs} & b\sigma(g_1)  }\} = X_{sxs}(b)
\]
has fibers of dimension $\frac 12 (\ell(sx) + \ell(s) - \ell(sxs)) = 1$, which proves the claim about $X_1$. Furthermore $X_2$ can be replaced with
\[
X_2' = \{ \xymatrix{ g_1 \ar[r]^-{sx} \ar@/_2pc/[rr]_-{sx} & g_2 \ar[r]^-{s} & b\sigma(g_1) } \}
\]
up to a universal homeomorphism, and $X_2'$ projects to $X_{sx}(b)$ with $1$-dimensional fibers. The corollary is proved.
\end{proof}

With slightly more care, one can show that in case~(2) of the lemma, the fibers of the projection $X_1\rightarrow X_{sxs}(b)$ are all isomorphic to $\mathbb A^1$, whereas the fibers of $X_2 \rightarrow X_{sx}(b)$ are $\mathbb A^1\setminus\{ 0\}$. This reflects the properties discussed at the end of subsection~\ref{subsec2.4}.

\begin{lemma}
Let $x, \t \in \tW$ with $\ell(\t)=0$. Then for any $b \in G(L)$, $X_{x}(b)$ is isomorphic to $X_{\t x \t \i}(b)$. 
\end{lemma}

\begin{proof}
Notice that $X_{x}(b)=\{g I; g \i b \s(g) \in I x I\}$. Thus the isomorphism $G(L)/I \to G(L)/I$, $g I \mapsto g I \t \i=g \t \i I$ gives an isomorphism from $X_{x}(b)$ to $\{g I; (g \t) \i b \s(g \t) \in I x I\}=\{g I; g \i b \s(g) \in \t I x I \s(\t) \i=I \t x \t \i I\}=X_{\t x \t \i}(b)$.
\end{proof}

Applying this lemma and the conjugation steps in the reduction method of Deligne and Lusztig, we obtain:

\begin{cor}\label{cor-reduction1}
Let $x, x' \in \tW$, $b\in G(L)$. If $x \to x'$, and $X_{x'}(b)\ne\emptyset$, then $X_x(b)\ne\emptyset$ and $\dim(X_{x}(b))-\dim(X_{x'}(b)) \ge \frac{1}{2}(\ell(x)-\ell(x'))$. 
\end{cor}

\subsection{}
In the sequel, we often use the following property of the Bruhat order: if $\alpha\in\Phi^+$ with corresponding reflection $s_\alpha$, and $w\in W$, then
\[
w s_\alpha > w \text{ if and only if } w\alpha > 0.
\]

\begin{lemma}\label{prelim-lemma}
Let $w,y\in W$ such that $w\alpha < 0$ for every $\alpha\in\Phi^+$ with $y^{-1}\alpha < 0$. Then $\ell(wy) = \ell(w)-\ell(y)$.
\end{lemma}

\begin{proof}
We proceed by induction on $\ell(y)$, the case $\ell(y)=0$ being clear. Write $y = sy'$, where $s$ is a simple reflection and $\ell(y')<\ell(y)$. Let $\alpha$ denote the simple root corresponding to $s$. One easily checks, using the above-mentioned property, that the pair $w' = ws$, $y'$ satisfies the induction hypothesis. Since $sy < y$ we have $y^{-1}\alpha < 0$, so $w\alpha < 0$ by assumption, and we obtain $w'=ws < w$. Altogether we have
\[
\ell(wy) = \ell(w'y') = \ell(w') - \ell(y') = \ell(w) -1-\ell(y') = \ell(w)-\ell(y).
\]
\end{proof}

\section{Proof of Reuman's conjecture}

\subsection{Outline of the proof}
\label{sec:outline}

We first state the result and give an outline of our strategy.
Throughout this chapter, we fix $b\in G(L)$, and we assume that whenever we consider $X_x(b)$, then $x$ and $b$ are in the same connected component of $G(L)$.


We consider the following maps from the extended affine Weyl group $\widetilde{W}$ to the finite Weyl group $W$:
\begin{align}
& \eta_1\colon \widetilde{W} = X_*(T) \rtimes W \rightarrow W, \text{ the projection}\\
& \eta_2, \text{ where $\eta_2(x)$ is the unique element $v$ such that $v^{-1}x\in {^S}\widetilde{W}$}\\
& \eta(x) = \eta_2(x)^{-1}\eta_1(x)\eta_2(x).
\end{align}

So if $x=vt^\mu w$ with $\mu$ dominant, $v\in W$, $w\in {}^{I(\mu)}W$, then $\eta_1(x)=vw$, $\eta_2(x)=v$, and $\eta(x) = wv$.
Furthermore, for $x\in \widetilde{W}$ (as always, in the same ``connected component'' as the fixed $b\in G(L)$) we define the \emph{virtual dimension}:
\[
d(x) = \frac 12 \big( \ell(x) + \ell(\eta(x)) -\defect(b)  \big).
\]
As discussed above, it is conjectured in \cite{GHKR2} that $\dim X_x(b)=d(x)$ for $b$ basic, $x\in \tW'$ with $X_x(b)\ne \emptyset$.

\begin{theorem}\label{thm1} 
Let $x\in \tW$ and assume that $\eta_2(x)=w_0$ or that the translation part of $x$ is given by a regular coweight. Then $\dim(X_x(b)) \le d(x)$.
\end{theorem}

\begin{theorem}\label{thm1'} 
Assume that $b$ is basic. Let $x\in\widetilde{W}'$ such that $\supp(\eta(x))=S$. Then $X_x(b) \ne \emptyset$.
\end{theorem}

This proves the non-emptiness statement in Conjecture~9.5.1~(a) of \cite{GHKR2}. Together with op.~cit., Proposition~9.5.4, which states that the converse of the theorem holds as well, this completely settles the emptiness versus non-emptiness question for basic $b$ and $x$ in the shrunken Weyl chambers $\tW'$.  The next theorem proves that the dimension of $X_x(b)$ is at least as large as predicted by the conjecture if $x\in\tW'$, and $G$ is a classical group and $b=1$ or $G$ is of type $A_n$:

\begin{theorem}\label{thm2}
(1) Let $G$ be a classical group, $x \in W_a$ such that $\supp(\eta(x))=S$. If moreover, $x \in \tW'$ or $\eta(x)$ is a Coxeter element of $W$, then $\dim X_x(1) \ge d(x)$.

(2) Let $G=PGL_n$ and $\t \in \widetilde{W}$ with $\ell(\t)=0$. Let $x \in W_a \t$ such that $\supp(\eta(x))$. If moreover, $x \in \tW'$ or $\eta(x)$ is a Coxeter element of $W$, then $\dim X_x(\t) \ge d(x)$.
\end{theorem}

The idea of the proofs of these theorems is to relate the given element $x$ to other elements for which non-emptiness, a lower bound on the dimension, or an upper bound on the dimension, respectively, are known. These relations will mainly be shown using the reduction method of Deligne and Lusztig. To this end, we introduce the following notation:

\begin{defn}
Let $x, y\in \widetilde{W}$ such that $x$, $y$ are in the same $W_a$-coset. We write $x\Rightarrow y$ if for every $b$,
\[
\dim X_x(b)-d(x)\ge \dim X_y(b)-d(y).
\]
\end{defn}

Here by convention, we set the dimension of the empty set to be $-\infty$. If the right hand side is $-\infty$ then the inequality holds regardless of the left hand side. In the definition (and in the theorem below) we do not assume that $b$ is basic. This is consistent with the expectation that whenever $x\in \tW'$ and $X_x(b)\ne\emptyset$, the difference $\dim X_x(b) - d(x)$ is a constant depending only on $b$, but not on $x$.

Note that this relation is transitive: If $x\Rightarrow y$, $y\Rightarrow z$, then $x\Rightarrow z$. By definition, if $x\Rightarrow y$ and $X_y(b)\ne\emptyset$, then $X_x(b)\ne\emptyset$. In this case, the lower bound $\dim X_y(b)\ge d(y)$ implies the analogous bound for $x$, while the validity of the upper bound $\dim X_x(b)\le d(x)$ implies the corresponding statement for $y$. We prove the following statements about the relation $\Rightarrow$:

\begin{theorem}\label{thm3}
\begin{enumerate}
\item 
Let $\mu$ be a dominant coweight, $v \in W$ and $w \in {}^{I(\mu)} W$. Assume that $v=w_0$ or that $\mu$ is regular. Then
\[
w_0t^\mu \Rightarrow v t^\mu w.
\]
\item
Let $a\in W$ with $\supp(a)=S$, and let $\mu\ne 0$ be a dominant coweight. Then there exists a Coxeter element $c\in W$ such that 
\[
a t^\mu \Rightarrow t^\mu c.
\]
\item
Assume that $x \in \tW'$, and that $\supp(\eta(x))$. Then there exist a dominant coweight $\lambda$ and $a \in W$ with $\supp(a)=S$ such that
\[
x \Rightarrow a t^\lambda.
\]
\item
Assume that $G$ is a classical group and $x \in W_a$ with $\eta(x)$ a Coxeter element of $W$, then 
\[ 
x \Rightarrow \eta(x).
\]
\end{enumerate}
\end{theorem}

Now the non-emptiness statement in Theorem~\ref{thm1'} follows from Theorem~\ref{thm3}~(2) \& (3) and the following lemma (Lemma 9.3.3 in \cite{GHKR2}), because Coxeter elements obviously are cuspidal.

\begin{lemma}
\label{Lemma-ADLVForCuspidal}
Let $\mu\in Y$, and let $w\in W$ be a cuspidal element (i.~e.~the conjugacy class of $w$ does not meet any standard parabolic subgroup), let $x  =t^\mu w$, and let $b$ be basic with $\kappa_G(b)=\kappa_G(x)$. Then $x$ is $\sigma$-conjugate to $b$, and in particular $X_x(b) \ne \emptyset$.
\end{lemma}

The upper bound on the dimension stated in Theorem~\ref{thm1} follows from Theorem~\ref{thm3}~(1) and the following lemma:

\begin{lemma}
\label{Lemma-UpperBoundFromGrass}
Let $x = w_0t^\mu$, where $\mu$ is a dominant coweight, and $w_0$ is the longest element in $W$. Then $\dim X_x(b) \le d(x)$.
\end{lemma}

\begin{proof}
By the dimension formula for affine Deligne-Lusztig varieties in the affine Grassmannian (see \cite{GHKR1}, \cite{Viehmann-dim}), we have
\[
\dim X_\mu(b) = \langle \rho,\mu-\nu_b\rangle -\frac12\defect(b),
\]
where $X_\mu(b)$ denotes the affine Deligne-Lusztig variety in the affine Grassmannian, and $\nu_b$ denotes the (dominant) Newton vector of $b$. Since $b$ is basic, its Newton vector is central and hence does not actually contribute anything. On the other hand, denoting by $\pi\colon \Flag\rightarrow\Grass$ the projection, we have
\[
\pi^{-1}(X_\mu(b)) = \bigcup_{x\in Wt^\mu W} X_x(b).
\]
Therefore, for all $w_1, w_2\in W$
\begin{align*}
\dim X_{w_1t^\mu w_2}(b)& \le \dim X_\mu(b) + \dim (G/B) = \langle \rho,\mu\rangle -\frac12\defect(b)+\ell(w_0)\\
& = \frac 12( \ell(w_0t^\mu) + \ell(w_0) -\defect(b) ) = d(w_0t^\mu).
\end{align*}
\end{proof}

To prove the lower bound under the additional assumptions in Theorem~\ref{thm2} (1), we reduce to an element of the finite Weyl group. In fact, the following lemma (for $\tau={\rm id}$) shows that it suffices to prove that $x\Rightarrow c$ for some element $c\in W$. If $\eta(x)$ is a Coxeter element, then this follows immediately from Theorem~\ref{thm3}~(4). On the other hand, suppose $x\in\tW'$ and $\supp(\eta(x))=S$. We apply Theorem~\ref{thm3}~(3). If the coweight $\lambda$ is $=0$, then we are done. Otherwise, we can use Theorem~\ref{thm3}~(2) to see that there exists a dominant coweight $\l$ and a Coxeter element $c\in W$ such that $x \Rightarrow t^\l c$. Writing $t^\lambda c = v_1t^\nu v_2$ with $t^\nu v_2\in {^S}\tW$, i.~e.~$\eta_2(t^\lambda c)=v_1$, we have $c=v_1v_2$ and $\eta(t^\lambda c)=v_2v_1$. Since $c=v_1 v_2$ is simply the decomposition into an element of $W_{I(\mu)}$ and an element of ${}^{I(\mu)}W$, we have $\ell(v_1)+\ell(v_2)=\ell(c)$ and since $c$ is a Coxeter element, $v_2 v_1$ is also a Coxeter element of $W$. Therefore $x \Rightarrow t^\l c \Rightarrow v_2 v_1$ (using Theorem~\ref{thm3}~(4)).

\begin{lemma}
Let $\t \in \widetilde{W}$ with $\ell(\t)=0$. Let $J \subset S$ with $\t(J)=J$. Then for any $w \in W_J$, $\dim X_{w \t}(\t)=\ell(w)$. 
\end{lemma}

\begin{proof}
By \cite[Lemma 9.7]{He}, $\dim X_{w \t}(\t)=\dim X_{\t}(\t)+\ell(w)$. By \cite[Prop 10.3]{He}, $\dim X_{\t}(\t)=0$. So $\dim X_{w \t}(\t)=\ell(w)$. 
\end{proof}

Under the assumption in Theorem~\ref{thm2} (2), we have that $\t=1$ or $0<r<n$ and that $\t$ is the length $0$-element that corresponds to the $r$-th fundamental coweight of $G$. The case that $\t=1$ is included in Theorem~\ref{thm2} (1). So we only need to consider the latter case. Similarly to the proof above, we have that $x \Rightarrow t^\l c$ for some Coxeter element $c$ of $W$. Let $m=\gcd(n, r)$. Then by \cite[Prop 6.7 (2)]{He}, $t^\l c \tilde \to (1 2 \cdots m) \t$. Hence by Corollary~\ref{cor-reduction1} and the Lemma above, \begin{align*} \dim X_{t^\l c} (\t) & \ge \dim X_{(1 2 \cdots m) \t}(\t)+\frac{1}{2}(\ell(t^\l c)-\ell((1 2 \cdots m) \t)) \\ &=\frac{1}{2}(\ell(t^\l c)+m-1).\end{align*} 

Since $c$ is a Coxeter element, $\eta(t^\l c)$ is also a Coxeter element of $W$. We also have that $\defect(\t)=n-m$. Thus $d(t^\l c)=\frac{1}{2}(\ell(t^\l c)+n-1-\defect(\t))=\frac{1}{2}(\ell(t^\l c)+m-1) \le \dim X_{t^\l c} (\t)$. So we obtain $d(x) \le \dim X_{x}(\t)$.




Therefore it remains to prove Theorem~\ref{thm3}. This is the goal of the following sections.

\subsection{Reduction of virtual dimension}

\begin{lemma}\label{X3}
If $x, x'\in\tW$ such that $x\to x'$ and $\ell(\eta(x))=\ell(\eta(x'))$, then $x\Rightarrow x'$.
\end{lemma}

\begin{proof}
Since $\ell(\eta(x))=\ell(\eta(x'))$, we have $d(x)-d(x')=\frac 12 (\ell(x)-\ell(x'))$. The Lemma now follows immediately from Corollary~\ref{cor-reduction1}.
\end{proof}

\begin{lemma}\label{Lemma-VirtualDim2'}
Let $x\in \widetilde{W}$. Let $s\in S$ be a simple reflection such  that $\ell(sxs) = \ell(x) - 2$. 
Then
\[
d(x) \ge d(sx) + 1,
\]
and equality holds if and only if $\ell(\eta(sx)) = \ell(\eta(x))-1$.
\end{lemma}

\begin{proof}
We write $x$ as $v t^\mu w$ with $v, w \in W$ and $\mu$ a dominant coweight such that $t^\mu w \in {}^S \tW$. Then $\eta(x) = wv$.

Since $s x<x$, we must have that $s v<v$. If $w s<w$, then $t^\mu w s \in {}^S \tW$ and $x s>x$, which is a contradiction. Therefore $w s>w$. Let $\alpha$ denote the simple root
corresponding to $s$, and write $\beta = v^{-1}(-\alpha)$, which is a positive
root because $sv<v$. We then have $wv(\beta) = w(-\alpha) < 0$ (since
$ws>w$), and obtain
\[
\eta(sx) = wsv = wvs_\beta < wv,
\]
as desired.
\end{proof}

Similarly, 

\begin{lemma}\label{Lemma-VirtualDim2''}
Let $x=v t^\mu w$ with $v, w \in W$ and $\mu$ a dominant coweight such that $t^\mu w \in {}^S \tW$. Let $s\in S$ be a simple reflection such that $\ell(sxs) = \ell(x) - 2$ and suppose that $t^\mu w s \in {}^S \tW$ or $w=1$. 
Then
\[
d(x) \ge d(xs) + 1,
\]
and equality holds if and only if $\ell(\eta(xs)) = \ell(\eta(x))-1$.
\end{lemma}

These results about the virtual dimension imply

\begin{lemma}\label{X2}
Let $x\in \tW$, $s\in S$ such that $\ell(sxs)<\ell(x)$. Then 
\begin{enumerate}
\item 
If $\ell(\eta(sx))=\ell(\eta(x))-1$, then $x\Rightarrow sx$.
\item
If $\ell(\eta(xs))=\ell(\eta(x))-1$, then $x\Rightarrow xs$.
\end{enumerate}
\end{lemma}

\begin{proof}
For (1), we simply use the Deligne-Lusztig reduction (where we consider $X_2$ in Corollary~\ref{Lemma-DLReduction}~(2)), and Lemma~\ref{Lemma-VirtualDim2'}. For part~(2), we first use the Deligne-Lusztig reduction from $x$ to $sx$ as in the first case. Then we use Corollary~\ref{Lemma-DLReduction}~(1) to reduce to $xs=s(sx)s$ which has the same length as $sx$. Altogether we see that if $X_{xs}(b)\ne\emptyset$, then $X_x(b)\ne\emptyset$, and then $\dim X_x(b)-\dim X_{xs}(b) \ge 1 = d(x)-d(xs)$.
\end{proof}

\subsection{Proof of Theorem~\ref{thm3}~(1)} We write $x = vt^\mu w$. 
First consider the case $v=\eta_2(x)=w_0$. Then we have that $w_0 t^\mu \Rightarrow x = w_0 t^\mu w$ because we can successively apply Lemma~\ref{X2}~(2).

Now we consider the case that $\mu$ is regular. 
Since $\mu$ is regular, $t^\mu wvw_0\in {^S}\tW$, so we can apply the ``$\eta_2=w_0$''-case to the element $w_0 t^\mu wvw_0$ and obtain that $w_0t^\mu\Rightarrow w_0 t^\mu wvw_0$. 
Because $\mu$ is regular, Lemma~\ref{X3} shows that 
\[ 
w_0 t^\mu wvw_0 = w_0 v\i  (vt^\mu w) vw_0 \Rightarrow vt^\mu w = x. 
\]

\subsection{Proof of Theorem~\ref{thm3}~(2)}

We prove the following stronger result:

Let $J \subset S$ and $x=v t^\mu w$ with $v,w\in W$, $\supp(v)=J$, $w$ is a Coxeter element in $W_{S-J}$, $\mu \neq 0$ and $t^\mu w \in {}^S \tW$. Then there exists a Coxeter element $c$ of $W$ such that $v t^\mu w \Rightarrow t^\mu c$. 

We proceed by induction on $|J|$. Suppose that the statement is true for all $J'\subsetneq J$, but not true for $J$. We may also assume that the claim of the proposition is true for all $v'$ with support $\supp(v')=J$ and $\ell(v')<\ell(v)$. Let $v=s_{i_1} \cdots s_{i_k}$ be a reduced expression. 

If $t^\mu w s_{i_1} \in {}^S \tW$, then $\ell(t^\mu w s_{i_1})=\ell(t^\mu)-\ell(w s_{i_1})=\ell(t^\mu)-\ell(w)-1=\ell(t^\mu w)-1$ and $\ell(s_{i_1} v t^\mu w s_{i_1})=\ell(s_{i_1} v)+\ell(t^\mu w s_{i_1})=\ell(s_{i_1} v)+\ell(t^\mu w)-1=\ell(v t^\mu w)-2$. By Lemma~\ref{X2}~(1), $v t^\mu w\Rightarrow s_{i_1} v t^\mu w$. If $\supp(s_{i_1} v)=J$, then by induction, there exists a Coxeter element $c$ of $W$ such that $s_{i_1} v t^\mu w \Rightarrow t^\mu c$. Hence $v t^\mu w\Rightarrow t^\mu c$. That is a contradiction.  

Now suppose that $t^\mu w s_{i_1} \in {}^S \tW$, but $\supp(s_{i_1}v)\subsetneq J$. In that case, we have $\ell(s_{i_1}vt^\mu ws_{i_1}) = \ell(vt^\mu w)-2$. So $vt^\mu w\Rightarrow s_{i_1}vt^\mu ws_{i_1}$ by Lemma~\ref{X3}. That is also a contradiction by induction hypothesis.  

Now we can assume that $t^\mu w s_{i_1} \notin {}^S \tW$. Then we have that $t^\mu w s_{i_1}=s_{i'_1} t^\mu w$ for some $i'_1 \in S$. So $w s_{i_1} w \i=t^{-\mu} s_{i'_1} t^\mu$ is a reflection in $W$. So $t^\mu$ commutes with $s_{i'_1}$ and $w s_{i_1} w \i=s_{i'_1}$ is a simple reflection. By our assumptions on $w$, it follows that $i'_1=i_1$ and $t^\mu w$ commutes with $s_{i_1}$. In this case, if $\ell(s_{i_1} v s_{i_1})<\ell(v)$, then $\supp(s_{i_1} v)=J$ and using Lemma~\ref{X2}~(1) and the induction hypothesis as in the first case, we again have that  $v t^\mu w \Rightarrow t^\mu c$ for some Coxeter element $c$ of $W$, which is a contradiction. 

Therefore we must have that $s_{i_1}$ commutes with $t^\mu w$ and $\ell(s_{i_1} v s_{i_1})=\ell(v)$. In this case, $v t^\mu w \approx s_{i_1} v s_{i_1} t^\mu w$. So $\dim X_{v t^\mu w} (b)=\dim X_{s_{i_1} v s_{i_1} t^\mu w} (b)$ by Corollary~\ref{Lemma-DLReduction}~(1). We also have $d(v t^\mu w)=d(s_{i_1} v s_{i_1} t^\mu w)$, because $\eta(s_{i_1} v s_{i_1} t^\mu w) = ws_{i_1}vs_{i_1}$ has length $\ell(w)+\ell(s_{i_1}vs_{i_1})$ (use that $w$ is a Coxeter element in $W_{S-J}$). Applying the same argument to $s_{i_1} v s_{i_1} t^\mu w$ instead of $v t^\mu w$, we have that $s_{i_2}$ commutes with $t^\mu w$. Repeating the same procedure, one can show that $s_{i_j}$ commutes with $t^\mu w$ for all $1 \le j \le k$. In particular, $s_k$ commutes with $w$ for all $k \in J$. Since $G$ is quasi-simple, this is only possible if $J=\emptyset$ or $J=S$. If $J=\emptyset$, then $v=1$ and $w$ is a Coxeter element of $W$ and the statement automatically holds. If $J=S$, then $s_i$ commutes with $t^\mu$ for all $i \in S$. Thus $\mu=0$, which contradicts our assumption.

\subsection{Proof of Theorem~\ref{thm3}~(3)}

Let $x=v t^\mu w \in \tW'$ with $\mu \in Y_+$, $v \in W$, $w \in {}^{I(\mu)} W$. We first give the definition of the elements $\g$ and $a$ that we use. Let $J=\{i \in S; s_i w<w\}$. Since $w \in {}^{I(\mu)} W$, $J \cap I(\mu)=\emptyset$. Hence $\mu-\rho_J^\vee \in Y_+$. By definition, $\<\rho^\vee_J, \a_i\>-\d_{w \i \a_i}=0$ for any $i \in S$. Since $x\in\tW'$, we obtain $\<\mu-\rho_J^\vee, \a_i\>+\d_{v \a_i} \neq 0$ for any $i \in S$. 

Let $J'=I(\mu-\rho^\vee_J)$. Then $v \a_i<0$ for any $i \in J'$. Thus $v=v' w^0_{J'}$ for some $v' \in W^{J'}$. Here $w^0_{J'}$ is the largest element in $W_{J'}$. Now $w v=(w v') w^0_{J'}=w' z$ for some $w' \in W^{J'}$ and $z \in W_{J'}$. Define $\g \in Y_+$ and $y \in W^{I(\g)}$ by $\mu-\rho^\vee_J+(w') \i \rho^\vee_J =y \g$. Furthermore, we define $a = (y\i z) * (w' y)$. It has support $\supp(a) = S$ since $S = \supp(wv) \subseteq \supp(w' y) \cup \supp(y\i z)$.

We show that 

\textbf{(a)} $\ell(w' y)=\ell(w')-\ell(y)$.

Let $\a \in \Phi^+$ with $y \i \a<0$. Then $\<\g, y \i \a\> \le 0$. If $\<\g, y \i \a\>=0$, then $y \i \a \in \Phi^-_{I(\g)}$ and $\a=y (y \i \a) \in \Phi^-$. That is a contradiction. Hence $\<\mu-\rho^\vee_J+(w') \i \rho^\vee_J, \a\>=\<y \g, \a\>=\<\g, y \i \a\><0$. Since $\mu-\rho^\vee_J \in Y_+$, $\<\mu-\rho^\vee_J, \a\> \ge 0$. Thus $\<\rho^\vee_J, w'\a\>=\<(w')\i \rho^\vee_J, \a\> <0$ and $w' \a<0$. Since $w' \a<0$ for any $\a \in \Phi^+$ with $y \i \a<0$, Lemma~\ref{prelim-lemma} shows that $\ell(w' y)=\ell(w')-\ell(y)$. (a) is proved. 

Now set $x_1=v z \i t^{\mu-\rho^\vee_J} y$ and $x_2=y \i z t^{\rho^\vee_J} w$. Then $x=x_1 x_2$ and we claim that 

\textbf{(b)} $\ell(x)=\ell(x_1)+\ell(x_2)$.

By the proof of (a), $y \in {}^{J'} W$. In fact, if for any $j \in J'$, $y^{-1} \alpha_j<0$, then by the proof of (a), $w' \alpha_j<0$, which contradicts that $w' \in W^{J'}$. Hence $t^{\mu-\rho^\vee_J} y \in {}^S \tW$ and $\ell(x_1)=\ell(v z \i)+\ell(t^{\mu-\rho^\vee_J})-\ell(y)=\ell(t^{\mu-\rho^\vee_J})+\ell(v)-\ell(z)-\ell(y)$. Also $\ell(x_2)=\ell(y \i z)+\ell(t^{\rho^\vee_J})-\ell(w)=\ell(t^{\rho^\vee_J})+\ell(y)+\ell(z)-\ell(w)$. Thus $\ell(x_1)+\ell(x_2)=\ell(t^{\mu-\rho^\vee_J})+\ell(t^{\rho^\vee_J})+\ell(v)-\ell(w)=\ell(t^\mu)+\ell(v)-\ell(w)=\ell(x)$. (b) is proved. 

Now
\[
X_{x}(b)= \{ \xymatrix{ g \ar[r]^-x & b \sigma(g) } \}=\{  \xymatrix{ g \ar[r]^-{x_1} & g_1 \ar[r]^-{x_2} & b \sigma(g) }\}.
\]
Set
\begin{align*}
X_1= &  \{ \xymatrix{ g_1 \ar[r]^-{x_2} & g_2 \ar[r]^-{x_1} & b \sigma(g_1) } \} \\
\cong & \{ \xymatrix{ g_1 \ar[r]^-{y\i z} &  g_3 \ar[r]^-{t^{\rho^\vee_J} w} & g_2 \ar[r]^-{x_1} & b \sigma(g_1)  } \}.
\end{align*}
The map $(g, g_1) \mapsto (g_1, b \sigma(g))$ is a universal homeomorphism from $X_{x}(b)$ to $X_1$.
Let 
\[
X_2=\{ \xymatrix{  g_1 \ar[r]^-{y\i z} &  g_3 \ar[r]^-{t^{\rho^\vee_J} w} \ar@/_2pc/[rr]_-{w' y t^{\g}} & g_2 \ar[r]^-{x_1} & b \sigma(g_1)  } \} \subset X_1.
\]
Then we have that $\dim(X_{x}(b)) \ge \dim(X_2)$.

Now let
\[
X_3=\{ \xymatrix{ g_1  \ar[r]^-{y\i z} & g_3  \ar[r]^-{w' y t^\g} & b \s(g_1) }  \},
\]
and let $f: X_2 \to X_3$ be the projection map. Notice that $w' y t^\g=t^{\rho^\vee_J} w x_1$. Thus by \ref{generalized-DL-reduction} (1),
the map is surjective and each fiber is of dimension $\frac{\ell(t^{\rho^\vee_J} w)+\ell(x_1)-\ell(w' y t^\g)}{2}=\frac{\ell(x)-\ell(y \i z)-\ell(w' y)-\ell(t^\g)}{2}$. Hence

\textbf{(c)} $\dim(X_{x}(b)) \ge \dim(X_3)+\frac{\ell(x)-\ell(y \i z)-\ell(w' y)-\ell(t^\g)}{2}$.

Notice that
\[
X_3=\{ \xymatrix{ g_1 \ar[r]^-{y\i z} & g_3 \ar[r]^-{w' y} &  g_4 \ar[r]^-{t^\g} & b \s(g_1)  }\}.
\]
Recall that $a=(y \i z)*(w' y)$. We set
\begin{align*}
X_4=& \{ \xymatrix{ g_1 \ar[r]^-{y\i z} \ar@/_2pc/[rr]_-{a} & g_3 \ar[r]^{w' y} &  g_4 \ar[r]^-{t^\g} & b \s(g_1)  } \}, \\
X_5=& \{ \xymatrix{ g_1 \ar[r]^-a & g_4 \ar[r]^-{t^\g} & b \s(g_1) } \}.
\end{align*}

By \ref{generalized-DL-reduction} (2), $\dim(X_3) \ge \dim(X_4) \ge \dim(X_5)+\ell(y \i z)+\ell(w' y)-\ell(a)$. As we proved above, $\ell(y \i z)+\ell(w' y)=\ell(w')+\ell(z)=\ell(w v)$. Therefore $\ell(y \i z)+\ell(w' y)+\frac{\ell(x)-\ell(y \i z)-\ell(w' y)-\ell(t^\g)}{2}=\frac{\ell(x)+\ell(w v)-\ell(t^\g)}{2}$. So 

\textbf{(d)} $\dim(X_{x}(b)) \ge \dim(X_5)+\frac{\ell(x)+\ell(w v)-\ell(t^\g)}{2}-\ell(a)$. 

Notice that $\ell(a t^\g)=\ell(a)+\ell(t^\g)$. Thus the map $(g_1, g_4) \mapsto g_1$ gives an isomorphism $X_5 \cong X_{a t^\g}(b)$. 

If $X_{a t^\g}(b)\ne \emptyset$, then we obtain that $X_{x}(b)\ne\emptyset$, and
\[
\dim X_{x}(b) \ge \dim X_{at^\g}(b)  + d(x) - d(a t^\g),
\]
i.e., $x \Rightarrow a t^\g$.

\begin{example}
Let us see what the elements $a$, $\gamma$ are in the following two special cases:
\begin{enumerate}
\item
If $v = 1$, then the assumption that $x\in \tW'$ implies that $J'=\emptyset$. Therefore $w'=w$, $z=y=1$, so that $\g =\mu$ and $a=w$. The result in this case is that $t^\mu w\Rightarrow w t^\mu$.
\item
If $\mu$ is ``very regular'', then $I(\mu) = I(\mu - \rho_J^\vee + (w')\i \rho_J^\vee) = \emptyset$, so that $a = w' = wv$. In this case we obtain $v t^\mu w \Rightarrow wv t^\g$.
\end{enumerate}
\end{example}

\subsection{Proof of Theorem~\ref{thm3}~(4)}

Since $\eta(x)$ is a Coxeter element of $W$, by \cite[Prop 6.7 (1)]{He}, we have that $x \tilde \to \eta(x)$. By Lemma \ref{X3}, $x \Rightarrow \eta(x)$.

\section{Remarks on the critical strips}

\subsection{A sharpened criterion for non-emptiness}

We have seen that the non-emptiness of $X_x(b)$ for $b$ basic and $x\in \tW'$ can be decided by looking at $\eta(x) = \eta_2(x)^{-1}\eta_1(x)\eta_2(x)$. In fact, if $\supp(\eta(x))\ne S$ and the translation part of $x$ is non-trivial in the sense that it is different from the Newton vector of $b$, then $X_x(b)=\emptyset$ (\cite[Proposition 9.5.4]{GHKR2}). If $x\in \tW\setminus \tW'$, then the converse is not true anymore, but one could ask whether it would help to check whether $\supp(\eta')=S$ for additional elements $\eta'\in W$ (depending on $x$), and specifically one could try to replace $\eta_2(x)$ by a different element of $W$. The following proposition gives a result in this direction, using the notion of $P$-alcove introduced in \cite{GHKR2}. In the proposition, $\eta_2(x)$ is replaced by an element of the form $s_\alpha\eta_2(x)$, where $\alpha$ depends on $x$. Recall that whenever $x$ is a $P$-alcove sufficiently far away from the origin, then $X_x(b)=\emptyset$.

Of course, the proposition is of particular interest, if ${^x}I \cap U_\alpha = I \cap U_\alpha$ (in particular $x\notin\tW'$, and more specifically, $x$ lies in the ``critical strip'' attached to $\alpha$).

\begin{prop}
Let $x=vt^\mu w\in \widetilde{W}$, $w\in {}^{I(\mu)}W$, let $\alpha$ be a finite root such that ${^x}I \cap U_\alpha \subseteq I \cap U_\alpha$, and assume that $-v^{-1}\alpha$ is a simple root. If there exists $j\in S$ such that
\[
(s_\alpha v)^{-1} \eta_1(x) (s_\alpha v) \in  W_{S\setminus \{ j \}},
\]
then $x$ is a $P$-alcove for $P = {}^{s_\alpha v}P_0$, $P_0=M_0 N_0$ the standard parabolic subgroup whose Levi component $M_0$ is generated by $S\setminus \{ j \}$.
\end{prop}

\begin{proof}
As usual, we write $P=MN$ for the Levi decomposition of $P$, where $M$ is the Levi subgroup containing the fixed maximal torus. By definition of $P$, $x\in \tW_M$, so we have to show that ${}^x I\cap U_\beta \subseteq I$ for every root $\beta$ occurring in $N$. By assumption, $-v^{-1}\alpha$ is a simple root $\alpha_i$.

Denote by $U$ the unipotent radical of the fixed Borel subgroup of $G$. Since the alcove $v^{-1}x$ lies in the dominant chamber, we have ${}^{v^{-1}x}I\cap U \subseteq I$. Furthermore, by our normalization of $I$ with respect to the dominant chamber, we have ${}^v (I\cap U) \subset I$, so altogether we obtain
\begin{equation}\label{prop41-eq1}
{}^x I\cap {}^v U \subset I.
\end{equation}

The set $R_N$ of roots occurring in $N$ is
\[
R_N = \{ v s_i \gamma =  s_\alpha v \gamma;\  \gamma\text{ a root in $N_0$} \}.
\]
We distinguish two cases: If $i\ne j$, i.e., $s_i\in W_M$, then $s_i$ stabilizes the set of roots in $N_0$, and therefore $R_N$ is the set of roots in ${}^v N_0$. In this case our claim follows from~\eqref{prop41-eq1}.

Now let us consider the case $i=j$, so that $\alpha_i\in R_{N_0}$, and $\alpha = -v \alpha_i = v s_i \alpha_i\in R_N$. For all $\beta\in R_N\setminus \{ \alpha \}$, we have $v^{-1}\beta > 0$, so ${}^x I\cap U_\beta \subset I$ by \eqref{prop41-eq1}, and finally we have ${}^x I\cap U_\alpha \subset I$ by assumption.
\end{proof}

Even though the proposition yields examples of pairs $(x,b)$ for which $X_x(b)= \emptyset$ although $\supp(\eta(x))=S$, it does not give rise to a sufficient criterion for non-emptiness in the critical strips, as can be shown by examples for $G=SL_4$.

\section{An example of an affine Deligne-Lusztig variety which is not equidimensional}
\label{Example-non-equidim}

\subsection{}
Looking again at the reduction method of Deligne and Lusztig, Corollary~\ref{Lemma-DLReduction}, we see that the situation in the affine case, in the shrunken Weyl chambers, is very much different from the classical situation: Whereas in the classical situation we always have $\dim X_1 = \dim X-1$, $\dim X_2=\dim X$ (denoting by $X$ the pertaining Deligne-Lusztig variety), in the affine shrunken case for the expected dimensions we have $d(sxs)+1=d(x)$, $d(sx)+1 \le d(x)$ --- so the closed part $X_1$, if non-empty, always should have the same dimension as the affine Deligne-Lusztig variety itself, while the open part has at most this dimension. See the example below for a specific case where this inequality is strict and where one can produce examples of affine Deligne-Lusztig varieties which are not equidimensional.

\subsection{}
To give an example of a non-equidimensional affine Deligne-Lusztig variety, we again use Corollary~\ref{Lemma-DLReduction}~(2). Let $x = v t^\mu w\in W_a$ be an element with $\mu$ dominant and very regular, and let $s\in S$ such that
\begin{enumerate}
\item
$\ell(sv) < \ell(v)$, $\ell(ws)>\ell(w)$,
\item
$\ell(wsv)< \ell(wv)-1$,
\item
$\supp(wv)=\supp(wsv)=S$.
\end{enumerate}
By~(1), we have $\ell(sxs) = \ell(x)-2$. We also have $\eta(x) =\eta(sxs)= wv$, $\eta(sx) = wsv$. Therefore as in Corollary~\ref{Lemma-DLReduction}~(2), we write $X_x(1) = X_1\cup X_2$, where $X_1$ is of relative dimension $1$ over $X_{sxs}(1)$, and $X_2$ is of relative dimension $1$ over $X_{sx}(1)$.

By the main result and assumption~(3), we have
\[
\dim X_2 = \dim X_{sx}(1) +1 = d(sx) + 1 < d(x) = \dim X_{x}(1),
\]
where the $<$ in the second line holds because of~(2) and Lemma~\ref{Lemma-VirtualDim2'}. (We also see that $\dim X_1=\dim X$.) Since $X_2$ is open in $X_x(1)$ and has strictly smaller dimension, $X_x(1)$ cannot be equidimensional.

It remains to find elements $v$, $w$, and $s$ that satisfy~(1)--(3). But this is easy, for instance for type $A_3$, we can take
\[
v = s_1s_2, \qquad w = s_1s_2s_3 s_2, \qquad s = s_1,
\]
so that
\[
wv = s_1s_2s_3s_2 s_1s_2,\qquad wsv = s_1s_2s_3.
\]


\begin{thebibliography}{MMM}

\bibitem{Beazley}
E.~Beazley, \emph{Non-emptiness for affine Deligne-Lusztig varieties associated to dominant Weyl chambers}, arXiv:1003.5969.

\bibitem{B}
R.~B\'edard, \emph{The lowest two-sided cell for an affine Weyl group} Comm. Algebra \textbf{16} (1988), 1113-1132.

%
%
%

\bibitem{DL} P.~Deligne, G.~Lusztig, \emph{Representations
  of reductive groups over finite fields},  Annals of Math.~{\bf 103}
  (1976), 103--161.  
%
%
%
\bibitem{GHKR1} U.~G\"{o}rtz, T.~Haines, R.~Kottwitz, D.~Reuman, {\em
Dimensions of some affine Deligne-Lusztig varieties}, Ann. sci.
\'Ecole Norm. Sup. 4$^e$ s\'erie, t. \textbf{39} (2006), 467--511.

\bibitem{GHKR2} U.~G\"{o}rtz, T.~Haines, R.~Kottwitz, D.~Reuman, {\em
Affine Deligne-Lusztig varieties in affine flag varieties}, to appear in Compositio Math., arXiv:0805.0045v4.

\bibitem{HartlViehmann} U. Hartl, E. Viehmann, \emph{The Newton stratification on deformations of local G-shtukas}, to appear in Journal f\"ur die reine und angewandte Mathematik, arXiv:0810.0821v3

\bibitem{H1}
X.~He, \emph{A subalgebra of $0$-Hecke algebra}, J. Algebra 322 (2009), 4030-4039.

\bibitem{H2}
X.~He, \emph{On some partitions of an affine flag variety}, arXiv:0910.5026. 

\bibitem{He} 
X.~He, {\em Minimal length elements in conjugacy classes of extended affine Weyl group}, arXiv:1004.4040v1.

\bibitem{He-WonderfulComp}
X.~He, \emph{Closure of Steinberg fibers and affine Deligne-Lusztig varieties},
arXiv:1003.0238

%


%
%
%
%
\bibitem{Kottwitz-defect} R. Kottwitz, \emph{Dimensions of {N}ewton strata
in the adjoint quotient of reductive groups},
Pure Appl. Math. Q.  \textbf{2} (2006), no. 3, 817--836.
%
%
%
%
%
%
%
%
%
%
%
%
%
\bibitem{L}
G.~Lusztig, \emph{Cells in affine Weyl groups}, Algebraic Groups and Related Topics. Adv. Studies Pure Math. 6  (1985), pp. 255--287. 

\bibitem{S}
J.~Y.~Shi, \emph{A two-sided cell in an affine Weyl group}, J. London Math. Soc. (2) 36 (3) (1987) 407--420.

\bibitem{Viehmann-dim} E.~Viehmann. {\em The dimension of affine Deligne-Lusztig varieties},  Ann. sci.
\'Ecole Norm. Sup. 4$^e$ s\'erie, t. \textbf{39} (2006), 513-526.
%
%
%
\end{thebibliography}
\end{document}